\documentclass[11pt,reqno]{amsart}
\usepackage{mathrsfs}

\font\bbbld=msbm10 scaled\magstephalf

\usepackage{hyperref}

\usepackage{graphicx}

\usepackage{xcolor}

\newcommand{\bpartial}{\bar{\partial}}

\def \p{\partial}

\newcommand{\bfR}{\hbox{\bbbld R}}

\newtheorem{theorem}{Theorem}[section]
\newtheorem{lemma}[theorem]{Lemma}

\newtheorem{corollary}[theorem]{Corollary}

 \theoremstyle{definition}

\theoremstyle{remark}

\numberwithin{equation}{section}

\begin{document}
\setlength{\baselineskip}{1.2\baselineskip}

\title[The general $J$-flows]
{The general $J$-flows}
\thanks{Project Funded by China Postdoctoral Science Foundation}

\author{Wei Sun}

\address{Shanghai Center for Mathematical Sciences, Fudan University,
         Shanghai 200433, China}
\email{sunwei\_math@fudan.edu.cn}

\begin{abstract}
We study the general $J$-flows. We use Moser iteration to obtain the uniform estimate. 
\end{abstract}

\maketitle

\section{Introduction}
\label{gjf-int}
\setcounter{equation}{0}
\medskip

In \cite{Donaldson99a}, Donaldson first defined the $J$-flow in the setting of moment maps. Later in the study of Mabuchi energy, Chen~\cite{Chen00b} independently defined the $J$-flow as the gradient flow for the $J$-functional under the normalization of the $I$-functional.

Let $(M,\omega)$ be a closed K\"ahler manifold of complex dimension $n \geq 2$, and $\chi$ a smooth closed real $(1,1)$ form in $\Gamma^k_\omega$. Throughout this paper, $\Gamma^k_\omega$ is the set of all the real $(1, 1)$ forms whose eigenvalue set with respect to $\omega$ belong to $k$-positive cone in $\bfR^n$. 
We consider the general $J$-flow for $ n \geq k > l \geq 1$,
\begin{equation}
\label{gjf-int-j-flow-equation}
	\frac{\p u}{\p t} = c -  \frac{\chi^l_u \wedge \omega^{n - l}}{\chi^k_u \wedge \omega^{n - k}},
\end{equation}
where
\begin{equation}
	c = \frac{\int_M \chi^l \wedge \omega^{n - l}}{\int_M \chi^k \wedge \omega^{n - k}} \qquad \text{and} \qquad \chi_u = \chi + \sqrt{- 1} \p\bpartial u .
\end{equation}
Indeed, the results of this paper apply to more general forms
\begin{equation}
	\frac{\p u}{\p t} = c -  \frac{\sum^{k - 1}_{l = 0} b_l \chi^l_u \wedge \omega^{n - l}}{\chi^k_u \wedge \omega^{n - k}},\quad b_l \geq 0 \text{ and } \sum^{k - 1}_{l = 0} b_l > 0 .
\end{equation}

We recall the general $J$-functionals, which was actually defined by Fang, Lai and Ma~\cite{FLM11}. Let $\mathcal{H}$ be the space 
\begin{equation}
	\mathcal{H} := \{ u\in C^\infty(M) \;|\; \chi_u \in \Gamma^k_\omega\} .
\end{equation}
For any curve $v(s) \in \mathcal{H}$, we define the funtional $J_m$ by
\begin{equation}
\label{uniform-J-functional-defition-derivative}
	\frac{d J_m}{d s} = \int_M \frac{\partial v}{\partial s} \chi^m_v \wedge \omega^{n - m} 
\end{equation}
for any $0 \leq m \leq k$. 
The parabolic flow~\eqref{gjf-int-j-flow-equation} can thus be viewed as the negative gradient flow for the $J_l$-functional under the normalization $J_k (u) = 0$.

It is easy to see that the $J$-flow is a special case,
\begin{equation}
	\frac{\p u}{\p t} = c -  \frac{\chi^{n - 1}_u \wedge \omega}{\chi^n_u} .
\end{equation}
Chen~\cite{Chen04} proved the long time existence of the solution to the $J$-flow. 
Weinkove~\cite{Weinkove04}~\cite{Weinkove06} showed the convergence under a strong condition.
In~\cite{SW08}, Song and Weinkove put forward and established a necessary and sufficient condition for convergence, which is called the cone condition.
A numerical version of the cone condition, which is easier to check in concrete examples, was proposed by Lejmi and Sz\'ekelyhidi~\cite{LejmiSzekelyhidi13}. 
Later, Collins and Sz\'ekelyhidi~\cite{CollinsSzekelyhidi2014a} affirmed the numerical cone condition on toric manifolds.

In this paper, we prove the following theorem.
\begin{theorem}
\label{main-theorem}
Let $(M,\omega)$ be a closed K\"ahler manifolds of complex dimension $n \geq 2$ and $\chi$ a closed form in $\Gamma^k_\omega$. Suppose that there exists $\chi' \in [\chi] \cap \Gamma^k_\omega$ satisfying the cone condition
\begin{equation}
\label{cone-condition-1}
	c k \chi'^{k - 1} \wedge \omega^{n - k} > l \chi'^{l - 1} \wedge \omega^{n - l} .
\end{equation}
Suppose that $u$ is the solution to the general $J$-flow~\eqref{gjf-int-j-flow-equation} on maximal time $[0,T)$. Then here exists a uniform constant $C>0$ such that for any $t \in [0,T)$
\begin{equation}
	\sup_M u (x,t) - \inf_M u (x,t) < C.
\end{equation}
\end{theorem}
We shall apply the Moser iteration approach in \cite{Sun2014e} and an idea from Blocki~\cite{Blocki2005a} and Sz\'ekelyhidi~\cite{Szekelyhidi2014b}. In order to obtain a time-independent uniform estimate, we apply the Moser iteration locally in time $t$. Since all other arguments are the same as those in \cite{Sun2015ph}, we immediately have a corollary. In \cite{Sun2015ph}, we applied the ABP estimate to obtain the uniform estimate.
\begin{corollary}
\label{main-corollary}
Under the assumption of Theorem~\ref{main-theorem}, there exists a long time solution $u$ to the general $J$-flow~\eqref{gjf-int-j-flow-equation}. 
Moreover, the solution $u$ converges in $C^\infty$ to $u_\infty$ with $\chi_{u_\infty} \in \Gamma^k_\omega$ satisfying
\begin{equation}
	\chi^l_{u_\infty} \wedge \omega^{n - l} = c \chi^k_{u_\infty} \wedge \omega^{n - k} .
\end{equation}
\end{corollary}

\bigskip

%
%
%
%

\section{The uniform estimate}
\label{uniform}
\setcounter{equation}{0}
\medskip

As in \cite{SW08}~\cite{FLM11}, the cone condition \eqref{cone-condition-1} is necessary and sufficient. In other words, the cone condition means that there is a $C^2$ function $v$ such that $\chi_v \in \Gamma^k_\omega$ and
\begin{equation}
\label{cone-condition-2}
	c k \chi^{k - 1}_v \wedge \omega^{n - k} >  l \chi^{l - 1}_v \wedge \omega^{n - l} .
\end{equation}
We may assume that there is $ c > \epsilon > 0$ such that
\begin{equation}
	(c - 2 \epsilon) k \chi^{k - 1}_v \wedge \omega^{n - k} >  l \chi^{l - 1}_v \wedge \omega^{n - l} .
\end{equation}
Without loss of generality, we may also assume that $\sup_M v = - 2 \epsilon$.


Along the solution flow $u(x,t)$ to equation~\eqref{gjf-int-j-flow-equation},
\begin{equation}
\label{uniform-theorem-kahler-decreasing-J-functional}
\begin{aligned}
	\frac{d}{dt} J_k (u) &= \int_M \frac{\partial u}{\partial t} \chi^k_u \wedge \omega^{n - l} \\
	&= c \int_M \chi^k_u \wedge \omega^{n - k}  - \int_M  \chi^l_u \wedge \omega^{n - l} = 0 .
\end{aligned}
\end{equation}
%

\begin{lemma}
\label{uniform-lemma-1}
At any time $t $,
\begin{equation}
\label{uniform-lemma-1-inequality}
	0 \leq \sup_M u \leq - C_1 \inf_M u + C_2 \;\;\;\text{ and }\;\;\; \inf_M  u(x,t) \leq 0.
\end{equation}
\end{lemma}

\begin{proof}
Choosing the path $v (s) = s u$ and noting that $J_k (u) = 0$,
\begin{equation}
\label{uniform-lemma-1-J-line}
	\frac{1}{k + 1} \sum^k_{i = 0} \int_M  u \chi^i_{ u} \wedge \chi^{k - i} \wedge \omega^{n - k} = 0.
\end{equation}
The first and third inequalities in \eqref{uniform-lemma-1-inequality} then follow from \eqref{uniform-lemma-1-J-line} and G\r{a}rding's inequality.

Rewriting \eqref{uniform-lemma-1-J-line},
\begin{equation}
	  \int_M  u  \chi^k  \wedge \omega^{n - k} = - \sum^k_{i = 1}  \int_M  u \chi^i_{ u} \wedge \chi^{k - i}  \wedge \omega^{n - k}.
\end{equation}
Let $C$ be a positive constant such that
\begin{equation}
	\omega^n \leq C \chi^k \wedge \omega^{n - l} .
\end{equation}
Then as in \cite{Sun2013p}, 
\begin{equation}
\label{uniform-lemma-1-inequality-1}
\begin{aligned}
	\int_M  u \omega^n &= \int_M \left( u - \inf_M  u\right) \omega^n + \int_M \inf_M  u \omega^n \\
	&\leq C \int_M \left( u - \inf_M  u\right) \chi^k \wedge \omega^{n - k} + \inf_M  u \int_M \omega^n \\
	&\leq \inf_M  u \left(\int_M \omega^n -  (k + 1) C \int_M  \chi^{n - \alpha }  \wedge \omega^\alpha\right) .
\end{aligned}
\end{equation}
The second inequality in \eqref{uniform-lemma-1-inequality} then follows from \eqref{uniform-lemma-1-inequality-1} (see Yau~\cite{Yau78}).

\end{proof}

\begin{proof}[Proof of Theorem~\ref{main-theorem}]
According to Lemma~\ref{uniform-lemma-1}, it suffices to prove a lower bound for $\inf_M (u - v) (x,t)$. We claim that
\begin{equation}
\label{uniform-theorem-lower-bound}
	\inf_M (u - v) (x,t) >  \inf_M u_t (x,0)  - C_0 ,
\end{equation}
where $C_0$ is to be specified later.

Differentiating the general $J$-flow~\eqref{gjf-int-j-flow-equation} with respect to $t$,
\begin{equation}
	\frac{\p u_t}{\p t} = \frac{\chi^l_u \wedge \omega^{n - l}}{\chi^k_u \wedge \omega^{n - k}} \Bigg(\frac{k \sqrt{- 1} \p\bpartial u_t \wedge \chi^{k - 1}_u \wedge \omega^{n - k}}{\chi^{k }_u \wedge \omega^{n - k}} - \frac{l \sqrt{- 1} \p\bpartial u_t \wedge \chi^{l - 1}_u \wedge \omega^{n - l}}{\chi^{l }_u \wedge \omega^{n - l}}\Bigg),
\end{equation}
which is also parabolic. Applying the maximum principle, $u_t$ reaches the extremal values at $t = 0$. So when $t \leq 1$, we have 
\begin{equation}
	\inf_M (u - v) (x,t) \geq \inf_M u(x,t) + 2 \epsilon \geq t \inf_M u_t (x,0) + 2 \epsilon .
\end{equation}
Therefore if \eqref{uniform-theorem-lower-bound} fails, there must be time $t_0 > 1$ such that
\begin{equation}
	\inf_M (u - v) (x,t_0) = \inf_{M \times [0,t_0]} (u - v) (x,t) = \inf_M u_t (x,0)  - C_0 \leq 0.
\end{equation}

For $p \geq 1 $, we consider the integral
\begin{equation}
\label{uniform-def-I}
\begin{aligned}
	I = \int_M \varphi^p \Big[\Big((c - u_t) \chi^k_u \wedge \omega^{n - k} - (c - \epsilon) \chi^k_v \wedge \omega^{n - k} \Big) - \Big( \chi^l_u \wedge \omega^{n - l} - \chi^l_v \wedge \omega^{n - l} \Big)\Big].
\end{aligned}
\end{equation}
It is easy to see that form some constant $C > 0$
\begin{equation}
	I \leq C \int_M \varphi^p \omega^n .
\end{equation}

\begin{equation}
	I = \int_M \varphi^p \Big[(c - \epsilon - s u_t + s\epsilon ) \chi^k_{s u + (1 - s) v} \wedge \omega^{n - k} - \chi^l_{s u + (1 - s) v} \wedge \omega^{n - l} \Big] \Bigg|^1_{s = 0} 
\end{equation}
For simplicity, we denote 
$
	\chi_s = \chi_{s u + (1 - s) v}.
$
Thus
\begin{equation}
\begin{aligned}
	I 
	&= \int^1_0 ds \int_M \varphi^p \sqrt{- 1} \p\bpartial (u - v) \wedge  \Big[k (c - \epsilon - s u_t + s\epsilon ) \chi^{k - 1}_s \wedge \omega^{n - k} - l \chi^{l - 1}_s \wedge \omega^{n - l} \Big]\\
	&\qquad - \int^1_0 ds \int_M \varphi^p (u_t - \epsilon) \chi^k_s \wedge \omega^{n - k} \\
	&= - p \int^1_0 ds \int_M \varphi^{p - 1} \sqrt{- 1} \p \varphi \wedge \bpartial (u - v) \\
	&\hspace{6em} \wedge  \Big[k (c - \epsilon - s u_t + s\epsilon ) \chi^{k - 1}_s \wedge \omega^{n - k}- l \chi^{l - 1}_s \wedge \omega^{n - l} \Big] \\
	&\qquad - \int^1_0 ds \int_M k s \varphi^p \sqrt{- 1} \bpartial (u - v) \wedge \p u_t \wedge \chi^{k - 1}_s \wedge \omega^{n - k} \\
	&\qquad - \int^1_0 ds \int_M \varphi^p (u_t -\epsilon) \chi^k_s \wedge \omega^{n - k} .
\end{aligned}
\end{equation}
We define
\begin{equation}
	\varphi = (u - v - \epsilon t - L )^- \geq 0,
\end{equation}
where $L$ to be specified later, and hence
\begin{equation}
\begin{aligned}
	I 
	 &= p \int^1_0 ds \int_M \varphi^{p - 1} \sqrt{- 1} \p (u - v) \wedge \bpartial (u - v) \\
	&\hspace{6em} \wedge \Big[k (c - \epsilon - s u_t + s\epsilon ) \chi^{k - 1}_s \wedge \omega^{n - k}- l \chi^{l - 1}_s \wedge \omega^{n - l} \Big] \\
	&\qquad - \int^1_0 ds \int_M k s \varphi^p \sqrt{- 1} \bpartial (u - v) \wedge \p u_t \wedge \chi^{k - 1}_s \wedge \omega^{n - k} \\
	&\qquad + \frac{1}{p + 1} \int^1_0 ds \int_M \p_t (\varphi^{p + 1} ) \chi^k_s \wedge \omega^{n - k} .
\end{aligned}
\end{equation}
Using integration by parts again, we know that almost everywhere over time $t$,
\begin{equation}
\begin{aligned}
	&\quad - \int^1_0 ds \int_M k s \varphi^p \sqrt{- 1} \bpartial (u - v) \wedge \p u_t \wedge \chi^{k - 1}_s \wedge \omega^{n - k} \\
	&= \frac{1}{p + 1} \int^1_0 ds \int_M k s \varphi^{p + 1} \sqrt{- 1} \p\bpartial u_t \wedge \chi^{k - 1}_s \wedge \omega^{n - k} \\
	&= \frac{1}{p + 1} \int^1_0 ds \int_M \varphi^{p + 1} \p_t (\chi^k_s \wedge \omega^{n - k}) ,
\end{aligned}
\end{equation}
and thus
\begin{equation}
\begin{aligned}
	I 
	&= p \int^1_0 ds \int_M \varphi^{p - 1} \sqrt{- 1} \p (u - v) \wedge \bpartial (u - v) \\
	&\hspace{6em} \wedge \Big[k (c - \epsilon - s u_t + s\epsilon ) \chi^{k - 1}_s \wedge \omega^{n - k}- l \chi^{l - 1}_s \wedge \omega^{n - l} \Big] \\
	&\qquad + \frac{1}{p + 1} \frac{d}{d t}\int^1_0 ds \int_M \varphi^{p + 1}  \chi^k_s \wedge \omega^{n - k} .
\end{aligned}
\end{equation}

Since $- S_{l - 1;i} / S_{k - 1;i}$ and $S^{\frac{1}{k - 1}}_{k - 1;i}$ are concave, we have
\begin{equation}
\begin{aligned}
	k (c - \epsilon - s u_t + s \epsilon) \chi^{k - 1}_s \wedge \omega^{n - k} - l \chi^{l - 1}_s \wedge \omega^{n - l} 
	&\geq (1 - s) \epsilon k \chi^{k - 1}_s \wedge \omega^{n - k} \\
	&\geq (1 - s)^k \epsilon k \chi^{k - 1}_v \wedge \omega^{n - k} .
\end{aligned}
\end{equation}
Then
\begin{equation}
\begin{aligned}
	I 
	&\geq 
	\frac{\epsilon k p}{k + 1} \int_M \varphi^{p - 1} \sqrt{- 1} \p (u - v) \wedge \bpartial (u - v) \wedge  \chi^{k - 1}_v \wedge \omega^{n - k} \\
	&\qquad + \frac{1}{p + 1} \frac{d}{d t}\int^1_0 ds \int_M \varphi^{p + 1}  \chi^k_s \wedge \omega^{n - k} .
\end{aligned}
\end{equation}

Integrating $I$ from $t_0 - 1$ to $t'\in [t_0 - 1, t_0]$, we obtain
\begin{equation}
\begin{aligned}
	C \int^{t'}_{t_0 - 1} dt \int_M \varphi^p \omega^n &\geq \sigma p \int^{t'}_{t_0 - 1} dt \int_M \varphi^{p - 1} \sqrt{- 1} \p (u - v) \wedge \bpartial (u - v) \wedge \omega^{n - 1} \\
	&\qquad + \frac{1}{p + 1} \int^1_0 ds \int_M \varphi^{p + 1}  \chi^k_s \wedge \omega^{n - k} \Bigg|^{t'}_{t = t_0 - 1} .
\end{aligned}
\end{equation}
Choosing $L = \inf_{M \times [0, t_0]} (u - v) - \epsilon t_0 + \epsilon$,
\begin{equation}
\label{uniform-moser-value-1}
	\varphi(x,t_0 - 1) 
	= (u - v -  \inf_{M \times [0, t_0]} (u - v) )^- = 0 ,
\end{equation}
and
\begin{equation}
\label{uniform-moser-value-2}
	\sup_{M \times [0,t_0]} \varphi (x,t) = \sup_M \varphi(x,t_0) = (u(x_i, t_0) - v(x_i) - \epsilon  -  \inf_{M \times [0, t_0]} (u - v) )^- = \epsilon .
\end{equation}
So
\begin{equation}
\label{uniform-moser-inequality-1-1}
\begin{aligned}
	C \int^{t'}_{t_0 - 1} dt \int_M \varphi^p \omega^n &\geq \sigma p \int^{t'}_{t_0 - 1} dt \int_M \varphi^{p - 1} \sqrt{- 1} \p (u - v) \wedge \bpartial (u - v) \wedge \omega^{n - 1} \\
	&\qquad + \frac{1}{p + 1} \int^1_0 ds \int_M \varphi^{p + 1}  \chi^k_s \wedge \omega^{n - k} \Bigg|_{t = t'} .
\end{aligned}
\end{equation}
By integration by parts, we observe that
\begin{equation}
\label{uniform-moser-inequality-1-2}
\int^1_0 ds \int_M \varphi^{p + 1}  \chi^k_s \wedge \omega^{n - k} \geq \int^1_0 ds \int_M \varphi^{p + 1}  \chi^k_v \wedge \omega^{n - k} =  \int_M \varphi^{p + 1}  \chi^k_v \wedge \omega^{n - k} .
\end{equation}
Substituting \eqref{uniform-moser-inequality-1-1} into \eqref{uniform-moser-inequality-1-2}
\begin{equation}
\label{uniform-moser-inequality-1-3}
\begin{aligned}
	C \int^{t'}_{t_0 - 1} dt \int_M \varphi^p \omega^n &\geq \sigma p \int^{t'}_{t_0 - 1} dt \int_M \varphi^{p - 1} \sqrt{- 1} \p (u - v) \wedge \bpartial (u - v) \wedge \omega^{n - 1} \\
	&\qquad + \frac{1}{p + 1} \int_M \varphi^{p + 1}  \chi^k_v \wedge \omega^{n - k} \Bigg|_{t = t'} \\
	&\geq \frac{2 \sigma }{p + 1} \int^{t'}_{t_0 - 1} dt \int_M \sqrt{- 1} \p \varphi^{\frac{p + 1}{2}} \wedge \bpartial \varphi^{\frac{p + 1}{2}} \wedge \omega^{n - 1} \\
	&\qquad +  \frac{1}{p + 1} \int_M \varphi^{p + 1}  \chi^k_v \wedge \omega^{n - k} \Bigg|_{t = t'} .
\end{aligned}
\end{equation}
Consequently,
\begin{equation}
\label{uniform-moser-inequality-1-4}
\begin{aligned}
	C (p + 1) \int^{t_0}_{t_0 - 1} dt \int_M \varphi^p \omega^n &\geq 2 \sigma \int^{t_0}_{t_0 - 1} dt \int_M \sqrt{- 1} \p \varphi^{\frac{p + 1}{2}} \wedge \bpartial \varphi^{\frac{p + 1}{2}} \wedge \omega^{n - 1} \\
	&\qquad + \sup_{t \in [t_0 - 1, t_0]}\int_M \varphi^{p + 1}  \chi^k_v \wedge \omega^{n - k}  .
\end{aligned}
\end{equation}
By Sobolev inequality, for $\beta = \frac{n + 1}{n}$,
\begin{equation}
\label{uniform-moser-inequality-1-5}
	C (p + 1) \int^{t_0}_{t_0 - 1} dt \int_M \varphi^p \omega^n \geq \Bigg(\int^{t_0}_{t_0 - 1} dt \int_M \varphi^{(p + 1) \beta}  \Bigg)^{\frac{1}{\beta}}.
\end{equation}
We can then iterate $\beta \rightarrow \beta^2 + \beta \rightarrow \beta^3 + \beta^2 + \beta \rightarrow \cdots$ and obtain
\begin{equation}
\label{uniform-moser-inequality-iteration}
	p_m = \frac{\beta (\beta^{m + 1} - 1)}{\beta - 1}
\end{equation}
and
\begin{equation}
\label{uniform-moser-inequality-1-6}
	(\ln C - \ln \beta) + \ln p_{m + 1} + p_m \ln ||\varphi ||_{L^{p_m}} \geq \frac{p_{m + 1}}{\beta} \ln ||\varphi||_{L^{p_{m + 1}}} .
\end{equation}
From \eqref{uniform-moser-inequality-1-6},
\begin{equation}
\label{uniform-moser-inequality-1-7}
	\sum^q_{m = 0}  \frac{\ln C - \ln \beta}{\beta^m} + \sum^q_{m = 0} \frac{\ln p_{m + 1}}{\beta^m} + \beta \ln ||\varphi ||_{L^{\beta}} \geq \frac{p_{q + 1}}{\beta^{q + 1}} \ln ||\varphi ||_{L^{p_{q + 1}}} ,
\end{equation}
that is
\begin{equation}
\label{uniform-moser-inequality-1-8}
	\ln \Big(\frac{C}{\beta - 1} \Big) \sum^q_{m = 0}  \frac{1}{\beta^m} + \sum^q_{m = 0} \frac{\ln (\beta^{m + 1} - 1)}{\beta^m} + \beta \ln ||\varphi||_{L^{\beta}} 
	\geq \frac{\beta^{q + 2} - 1}{\beta^q (\beta - 1)} \ln ||\varphi ||_{L^{p_{q + 1}}} .
\end{equation}
Letting $q \rightarrow \infty$,
\begin{equation}
\label{uniform-moser-inequality-1-9}
	C +  \ln ||\varphi||_{L^{\beta}} \geq \frac{\beta}{\beta - 1} \ln ||\varphi||_{L^\infty} = \frac{\beta \ln \epsilon}{\beta - 1} .
\end{equation}
Therefore, there exists a uniform constant $c_1 > 0$ such that
\begin{equation}
\label{uniform-moser-inequality-1-10}
	\int^{t_0}_{t_0 - 1} dt \int_M \varphi^\beta \omega^n \geq c_1 .
\end{equation}
So we have
\begin{equation}
	\epsilon^\beta \int^{t_0}_{t_0 - 1} dt \int_{\{\varphi > 0\}} \omega^n \geq c_1 .
\end{equation}
When $\varphi > 0$,
\begin{equation}
	u < v + \epsilon t + \inf_{M \times [0,t_0]} (u - v) -\epsilon t_0 + \epsilon \leq v + \epsilon + \inf_{M \times [0,t_0]} (u - v) < \inf_{M \times [0,t_0]} (u - v)  .
\end{equation}
Thus, using an idea from Blocki~\cite{Blocki2005a} and Sz\'ekelyhidi~\cite{Szekelyhidi2014b},
\begin{equation}
\begin{aligned}
	c_1  
	&\leq \epsilon^\beta \int^{t_0}_{t_0 - 1} \frac{|| u^- (x,t)||_{L^1}}{|\inf_{M \times [0,t_0]} (u - v)|} dt \\
	&\leq \epsilon^\beta \int^{t_0}_{t_0 - 1} \frac{|| u(x,t) - \sup_M u(x,t)||_{L^1}}{|\inf_{M \times [0,t_0]} (u - v)|} dt .
\end{aligned}
\end{equation}
Since $\Delta u $ has a lower bound, we have a uniform bound for $||u(x,t) - \sup_M u(x,t)||_{L^1}$. 
As a consequence, there is a uniform constant $C_0 > 0$ such that
\begin{equation}
\inf_{M \times [0,t_0]} (u - v) > - C_0 . 
\end{equation}
However, it contradicts the definition of $t_0$.

\end{proof}

\bigskip

\noindent
{\bf Acknowledgements}\quad
The author is very grateful to Bo Guan for his advice and encouragement. The author also wishes to thank Hongjie Dong for helpful discussions.

\end{document}